%% file: ciconiametric.tex
\documentclass[11pt,a4paper]{article}
\usepackage{amsmath,amsthm,amssymb,amsfonts}
\usepackage{hyperref,graphicx}
\usepackage{mathtext}
\newcommand{\C}{{\mathbb{C}}}          
\newcommand{\R}{{\mathbb{R}}}          

\newcommand{\gll}{{\mathfrak{gl}}}       %
\newcommand{\XIS}{{\mathfrak{X}}}

\newcommand{\estrela}{{\boldsymbol{\star}}}

\newcommand{\SO}{{\mathrm{SO}}}
\newcommand{\Ort}{{\mathrm{O}}}
\newcommand{\SU}{{\mathrm{SU}}}
\newcommand{\Uni}{{\mathrm{U}}}

\newcommand{\rr}{\rightarrow}
\newcommand{\lrr}{\longrightarrow}

\newcommand{\calu}{{\cal U}}             %
\newcommand{\na}{{\nabla}}
\newcommand{\End}[1]{{\mathrm{End}}\,{#1}}

\newcommand{\dx}{{\mathrm{d}}}

\newcommand{\inv}[1]{{#1}^{-1}}
\newcommand{\papa}[2]{\frac{\partial#1}{\partial#2}}
\newcommand{\dpapa}[2]{\dfrac{\partial#1}{\partial#2}}

\newcommand{\ciconiametric}{{\mathrm{g}}}

\newcommand{\db}{\overline{\partial}}

\newcommand{\cz}{{\overline{z}}}
\newcommand{\cw}{{\overline{w}}}
\newcommand{\ceta}{{\overline{\eta}}}
\newcommand{\ca}{{\overline{a}}}
\newcommand{\cgamma}{{\overline{\Gamma}}}

\newcommand{\fhmenosaquadrado}{\vartriangle}

\newcommand{\cinf}[1]{{\mathrm{C}}^\infty_{#1}}

\newtheorem{teo}{Theorem}[section]
\newtheorem{lemma}{Lemma}[section]
\newtheorem{coro}{Corollary}[section]
\newtheorem{prop}{Proposition}[section]
\newtheorem*{propsemnumero}{Proposition}
\newtheorem{defi}{Definition}[section]

\def\cyclic{\mathop{\kern0.9ex{{+}
\kern-2.2ex\raise-.28ex\hbox{\Large\hbox{$\circlearrowright$}}}}\limits}

\title{The ciconia metric on the tangent bundle \\ of an almost-Hermitian manifold}

\author{R. Albuquerque}

\begin{document}

\maketitle

\begin{abstract}

We find a new class of invariant metrics existing on the tangent bundle of any given almost-Hermitian manifold. We focus here on the case of Riemannian surfaces, which yield new examples of K\"ahlerian Ricci-flat manifolds in four real dimensions.

\end{abstract}

\ 
\vspace*{4mm}
\vspace*{1mm}

\vspace*{7mm}

\markright{\sl\hfill  R. Albuquerque \hfill}

\input{newmetric}

\input{symplecticiconia}

\vspace*{5mm} 

\textsc{R. Albuquerque}\ \ \ \ \textbar\ \ \ \ 
{\texttt{rpa@uevora.pt}}

Departamento de Matem\'atica da Universidade de \'Evora

Centro de Investiga\c c\~ao em Mate\-m\'a\-ti\-ca e Aplica\c c\~oes

Rua Rom\~ao Ramalho, 59, 671-7000 \'Evora, Portugal

\ \\
The research leading to these results has received funding from Funda\c c\~ao para a Ci\^encia e a Tecnologia.

\newpage
\markright{\sl\hfill THIS PART NOT  FOR PUBLICATION \hfill}

\input{Proofs}

\end{document}

%% file: newmetric.tex
\section{A new invariant metric}

\subsection{Introduction}

This article brings to light a new geometric structure associated to any given almost-Hermitian manifold $(M,g,J)$.

We define an almost-Hermitian structure on $TM$, adding to the theory of the geometry of tangent bundles, for which the underlying metric generalizes both the Sasaki and the Yano metrics with weights. Indeed here a new invariant symmetric tensor is exhibited, `$g_a$', which combines with the well-known geometry of Riemannian fibre bundles.

The literature on similar structures does not refer our invariant construction. Comparison with recent studies on special metrics on tangent bundles, like those originating from $g$-natural, Calabi, Eguchi-Hanson, Gibbons-Hawking or Taub-NUT metrics, cf. \cite{AbbSarih,Blair,Calabi,DaviesYano,GibbonsHawking,KMS,KowSek0,Stenzel}, will show that a \textit{ciconia metric} stands quite unique in the field of Riemannian structures on vector bundles. For instance, the well-known $\SU(2)$-holonomy Stenzel metric on $T_{S^2}$, which is a special case of Eguchi-Hanson metric, cannot be realised as a ciconia metric since the zero-section is there a Lagrangian submanifold, cf. \cite{Stenzel}.

Any orientable 2-dimensional Riemannian manifold $M$, any two real functions $f,h$ and any complex function $a$ on $M$, such that $fh-|a|^2\neq0$, give rise to a new example. Therefore we have a new class of pseudo-Riemannian and Riemannian geometries. We concentrate on the lowest possible dimension, much more being due to be researched. The curvature of ciconia metrics, for example, stands as an open question in the general setting. It would be quite important to analyse the case of a K\"ahler manifold base.

\subsection{The diagonal structure of the tangent manifold}

\label{sec:Introduction}

It is widely known that the total space of the tangent bundle of any given $\mathrm{C}^k$ differentiable manifold $M$ carries the structure of a $\mathrm{C}^{k-1}$ differentiable manifold of twice the dimension of $M$. It is also very well known the existence of a pseudo-Riemannian structure on the same total space, the so-called Sasaki metric, when the base manifold is endowed with a pseudo-Riemannian structure.

The purpose of the present article is to introduce a more general invariant construction, which is natural to almost-Hermitian geometry in any dimension. We shall start our study later-on with the analysis of the 2-dimensional case.

Let $\pi:TM\lrr M$ denote the tangent bundle of a smooth Riemannian ma\-ni\-fold $(M,g)$. We denote by $T_M$ the total space of such vector bundle with structure group $\mathrm{GL}(m,\R)$, where $m=\dim M$. In order to swiftly present the main results, we let $\na$ denote the Levi-Civita connection of $M$ (though any other metric-connection would be interesting to consider as well). The natural vertical tangent bundle to $T_M$, this is, $V:=\ker\dx\pi$, admits a tautological section $U$ defined by $U_u=u\in T_{\pi(u)}M$. Then the well-esta\-bli\-shed theory proves two identifications; first, that $V=\pi^\estrela TM$ and, second, that ${H^{\na}}:=\ker(\pi^\estrela\na_\cdot U)$ defines a vector bundle and a complement to $V$. Indeed, we have, $\forall X\in TT_M$,
\[    \pi^\estrela\na_{X}U=X^v \]
(using the canonical notation for the projections). The horizontal distribution ${H^{\na}}$ is then also identified isomorphically with $\pi^*TM$, via $(\dx\pi)_|:{H^{\na}}\lrr\pi^*TM$, as vector bundles over $T_M$.

With the canonical splitting $TT_M:=T(T_M)={H^{\na}}\oplus V\simeq\pi^*TM\oplus\pi^\estrela TM$ comes the notion of a \textit{mirror map} $B$ of $TT_M$. Such is the endomorphism which sends a horizontal vector to the respective vertical lift and sends any vertical vector to 0. Then we have the notion of \textit{adapted} frame on $T_M$, which is a horizontal lift of a frame on $M$ together with or followed by its mirror in $V$.

Adapted frames on the manifold $T_M$ yield the structural group reduction $\mathrm{GL}(m,\R)\hookrightarrow\mathrm{GL}(2m,\R)$. The representation of the smaller subgroup follows from the {diagonal} inclusion in the larger. Such is the peculiar feature of the geometry of tangent vector bundles. An immediate consequence, e.g., is that the reduction to the diagonal subgroup is carried through $\mathrm{GL}_+(2m,\R)$. Hence the tangent manifold is always orientable, independently of $M$.

Recall $TM\lrr M$ is a vector bundle associated to the principal frame bundle $FM\lrr M$. Now, given the connection $\na$, the vector bundle $TT_M\lrr T_M$ can be associated to $\pi^*FM\lrr T_M$ with the same structure group: $TT_M=\pi^*FM\times_{\mathrm{GL}(m,\R)}(\R^m\oplus\R^m)$.

For any tensor field on $M$, we write $\pi^*$ to denote a pull-back or horizontal lift to $T_M$, and write $\pi^\estrela$ to denote a vertical lift. Regarding nomenclature, the vector field $U$ is also known as the Liouville vector field and $S=B^\mathrm{t} U$ as the geodesic spray.

We shall intensely use the following associated linear connection on $T_M$. We denote by $\na^*$ the direct-sum of the pull-back connections $\pi^*\na$ on both sides of the canonical splitting of $TT_M$. In particular, we find $\na^*B=0$. Indeed we obtain a linear connection on the manifold $T_M$, satisfying, $\forall X\in TT_M,\ Y\in\XIS_M$,
\[    \na^*_{X}\pi^*Y=\pi^*(\na_{\dx\pi(X)}Y), \qquad
\na^*_{X}\pi^\estrela Y=\pi^\estrela(\na_{\dx\pi(X)}Y) .\]


Now, introducing the Riemannian structure $(M,g)$, we may consider the frames on $M$ which are orthonormal, and proceed with a further reduction to the principal bundle which has $\Ort(m)$ as structure Lie group. A reduction of the original structure group of $TT_M$ follows, as before,
from the existence of adapted frames and from the natural \textit{diagonal} inclusion
\begin{equation}
  \Ort(m)\hookrightarrow\Ort(m)\times\Ort(m)\subset\mathrm{GL}(2m,\R) .
\end{equation}

Finally we are ready for the presentation of a new idea on $T_M$.

One may study product metrics on any total spaces of pseudo-Riemannian vector bundles over pseudo-Riemannian manifolds, of any rank, through `H+V' decomposition of their tangent bundle, cf. \cite{Alb2014d,Calabi}. However, in such cases the structure group corresponds in general with the \textit{product} of two Lie groups. Even so, as we shall do below, one may still include weights on horizontal and on vertical directions.

On tangent spaces, the so-called $g$-natural metrics have not ceased to being studied ever since the Sasaki metric was first found. References \cite{AbbSarih,Alb2014d,Blair,KMS,KowSek0} and others therein may guide the interested reader. Several Riemannian metrics of different types on $T_M$ have been discovered in the latest decades, in the breadth of ideas such as those described in \cite{GibbonsHawking}.



\subsection{The new almost-Hermitian metric}

\label{sec:Thenewmetric}

Regarding the reduction to $\Ort(m)$ on the tangent manifold, which has now become clear, a new metric structure is admissible without further imposing any restrictions on the base manifold. As explained earlier, we stress this new structure is impossible to reproduce on other vector bundle manifolds.

Taking any isomorphism $A\in\End{(TT_M)}$ symmetric for the canonical (Sasaki) metric, one may define another metric on the same manifold by $(\pi^*g\oplus\pi^\estrela g)(A\,\cdot\,,\,\cdot\,)$. In particular, we consider the following symmetric bilinear-form:
\begin{equation}\label{Definitionciconiametric}
 \ciconiametric_{f,a,h}=f\pi^*g+g_a+h\pi^\estrela g
\end{equation}
where $f$ and $h$ are real functions, $f,h:T_M\lrr\R$, and $g_a$ is defined by
\begin{equation}\label{Definitionmetric_g_a}
\begin{split}
  g_a(x^h,y^v) &=\pi^\estrela g(a(x^h),y^v) \\
  &=\pi^*g(x^h,a^\mathrm{t}(y^v)) ,
\end{split}
\end{equation}
$\forall u\in T_M$, $\forall x,y\in T_{\pi(u)}M$, with $a$ an endomorphism of $TT_M$ such that $a(x^v)=a^\mathrm{t}(x^h)=0$. Still there is more to this example, because it coincides with the general case.
\begin{prop}
Any Riemannian structure given by $(\pi^*g\oplus\pi^\estrela g)(A\,\cdot\,,\,\cdot\,)$ is compatible with the reduction to the structure group $\Ort(m)$ if and only if it is of the above type \eqref{Definitionciconiametric}.
\end{prop}
Indeed, the metric being defined independently of the choice of adapted frame is the same as $A$ being invariant for the diagonal representation. Equivalently, with respect to the canonical splitting of $TT_M$, there exists a vector bundle morphism $a:{H^{\na}}\lrr V$, commuting with the $\Ort(m)$-representation, and there exist functions $f,h$ as above such that
\begin{equation}\label{symmetriccommutewithdiagonalaction}
 A=\left[\begin{array}{cc} f1_m & a^\mathrm{t} \\
a & h1_m    \end{array}   \right].
\end{equation}
The result follows from the next Lemma, which further determines all possible $a$.
\begin{lemma}\label{Lemmadarepresentacao}
 A symmetric linear map $A$ from the Euclidean space $\R^m\oplus\R^m$ onto itself commutes with the diagonal representation of $\SO(m)$, respectively $\Ort(m)$, if and only if there exist $f,h\in\R$ and a linear map $a:\R^m\rr\R^m$ in the centraliser subgroup of $\SO(m)$, respectively $\Ort(m)$, in $\gll(m,\R)$ such that $A$ has the shape of \eqref{symmetriccommutewithdiagonalaction}.
 
 Moreover, in the case of $\Ort(m)$, then $a=b1_m$ for some $b\in\R$; and in the case of $\SO(m)$ then two cases are possible:\\
 (i) for $m=2$, we have $a=b+ic=\left[\begin{array}{cc} b & -c \\ c & b \end{array}\right]$ with $b,c\in\R$\ \,($i=\sqrt{-1}$);\\
 (ii) for $m>2$, we have $a=b1_m$ for some $b\in\R$.
 \end{lemma}
\begin{proof}
 The first part is trivial: writing $A$ as a four blocks symmetric matrix, then $A$ commutes with 
 $\left[\begin{array}{cc} o &  \\  & o    \end{array}   \right] ,\ \forall o\in\Ort(m)$,
 if and only if it is of the desired shape. For the second part, the centraliser subgroup of the orthogonal group in $\gll(m,\R)$ is easily determined from Lie algebra theory. The exception in the case oriented and $m=2$ is immediate.
\end{proof}
Clearly, the canonical or Sasaki metric corresponds with the bilinear-form $\ciconiametric_{1,0,1}$. And the identity map $a=1_m$ corresponds to the morphism $B$ restricted to horizontals.

\begin{prop}\label{prop_signatureofciconiametric}
Let $\fhmenosaquadrado=fh-|a|^2$. Then the metric $\ciconiametric_{f,a,h}$ is:\\
 (i) positive definite if and only if $f>0$ and $\fhmenosaquadrado>0$;\\
 (ii) negative definite if and only if $f<0$ and $\fhmenosaquadrado>0$;\\
 (iii) of signature $(m,m)$ if and only if $f<0$ and $\fhmenosaquadrado<0$, or $f=0$ and $a\neq0$, or $f>0$ and $\fhmenosaquadrado<0$; in other words, if and only if $\fhmenosaquadrado<0$.

No other signatures may occur.
\end{prop}
\begin{proof}
 Recall $|a|^2=b^2+c^2$.
 Notice for case $m>2$, then $c=0$ and thus, with $f\neq0$ or $h\neq0$, it is trivial to find the sequence of minors of $A$ in \eqref{symmetriccommutewithdiagonalaction}. That is $f,f^2,\ldots,f^m,\ldots,f^m(h-\frac{b^2}{f})^k=\frac{f^m\fhmenosaquadrado^k}{f^k},\ldots,\ \forall 1\leq k\leq m$. For $m=2$ and any $c$, the sequence of minors of $A$ is $f,f^2,f\fhmenosaquadrado,\fhmenosaquadrado^2$. Due to the shape of the matrix, the signature $(m,m)$ is the only remaining complementary to definite signature. If $f=h=0$ and $m>2$, the case is trivial. For $m=2$, we consider an adapted frame $u_1,u_2,u_3,u_4$ with $g_a(u_2,u_4)=g_a(u_1,u_3)=b,\ -g_a(u_2,u_3)=g_a(u_1,u_4)=c$ and all other products vanishing, and moreover, without loss of generality, we assume $b\neq0$. Then a $(2,2)$-orthonormal frame is found through the orthogonal frame $v_1=u_1+u_3,\ v_2=u_2+u_4$,
 \[ v_3=\frac{b-c}{2}(u_1+u_4)-\frac{b+c}{2}(u_2+u_3),\quad v_4=\frac{b+c}{2}(u_1-u_4)+\frac{b-c}{2}(u_2-u_3) . \]
 The identities $g_a(v_1,v_1)=g_a(v_2,v_2)=2b,\ g_a(v_3,v_3)=g_a(v_4,v_4)=-b(b^2+c^2)$ follow.
\end{proof}
Of all the metrics defined on a tangent manifold, the $g$-natural metrics are the most studied in the literature, cf. \cite{Blair,KMS,KowSek0}. For instance, recall the pseudo-Riemannian metric $\ciconiametric_{0,1_m,0}$ was essentially found by K.~Yano, cf. \cite{AbbSarih,DaviesYano,KMS,KowSek0}. The $g$-natural metrics in general do coincide with the metric $\ciconiametric_{f,a,h}$ when we do not enter in the details of case $m=2$; in general, for $m>2$, we have $a=b1_m$. Therefore, we shall not focus on the general dimension case.


We assume from now on that the manifold $M$ is orientable, so that we may start from an $\SO(m)$ structure, and thus consider case (i) in Lemma \ref{Lemmadarepresentacao}.
\begin{defi}
 Given the smooth functions $f,h,b,c:T_M\lrr\R$ and letting $a=b+ic$, the new $g$-natural metric from \eqref{Definitionciconiametric} shall be called a {\emph{ciconia}} metric.
\end{defi}%
The ciconia\footnote{We use the name ciconia (from the latin for stork) inspired by \eqref{xmetricmatrix} and the migrant bird species, abundant in three continents and believed to have an exceptional sense of orientation.} 
  metric matrix on an adapted orthonormal frame is given by
\begin{equation}\label{xmetricmatrix}
 \left[\begin{array}{cccc} f & & b &c \\  & f&-c & b \\
b & -c & h & \\ c & b & & h  \end{array}   \right].
\end{equation}

We notice that all the metrics considered are compatible with the canonical lift $\pi^*J\oplus\pi^\estrela J$ of any metric compatible almost-complex structure $J$ on $M$. This holds simply because $J$ is a vector bundle {isometry}! In particular, in dimension 2, the metric is compatible with natural 90 degree rotation. This last remark proves to be extremely helpful, as we shall see in the next section. 

The previous findings also show we may proceed with analogous definition of a ciconia metric on the total space of the tangent bundle of any given almost-Hermitian, i.e. $\Uni(m/2)$-manifold, for any even $m$. The shape of \eqref{xmetricmatrix} arising from a unitary frame remains, since the centraliser of the unitary group in $\gll(m,\R)$ is $\{b1_{m/2}+ic1_{m/2}: \  b,c\in\R\}$.
\begin{teo}
 The ciconia metric on the tangent manifold of any almost-Hermi\-ti\-an manifold is itself almost-Hermitian.
\end{teo}

Leaving aside the trivial endomorphisms $a$, the definition of the metric $\ciconiametric_{f,a,h}$ is entirely new in the literature, when $c\neq0$.

\subsection{Equations of ciconia metric}

\label{sec:Eotam}

The best way to grasp the geometry of ciconia metrics seems to be by referring to isothermal coordinates. Recall every Riemannian surface $M$ with H\"older continuous metric is locally conformal to the Euclidean plane, thus providing the isothermal coordinates. We proceed with these sufficiently general hypotheses, which include analytic spaces. Or rather let us assume we have a smooth metric $g$ on $M$ given locally by 
\begin{equation}
 g=\lambda\,\dx z\dx\cz ,
\end{equation}
for some real smooth function $\lambda>0$ and complex coordinates $z=x+iy$ on an open subset $\calu\subset M$.

We identify $T_zM$ with $T^{1,0}_zM=\C\partial_z$, both restricted to the set $\calu$. This linear map is defined by $\partial_x\longmapsto\partial_z,\ \partial_y\longmapsto i\partial_z$ where we denote $\partial_x=\papa{}{x}$ and $\partial_z=\papa{}{z}=\frac{1}{2}(\partial_x-i\partial_y)$. We shall also require $\partial_\cz=\overline{\partial_z}$. Clearly, the metric $g$ satisfies
\begin{equation}
 \lambda=g(\partial_x,\partial_x)=g(\partial_y,\partial_y)=2g(\partial_z,\partial_\cz).
\end{equation}

With the usual abbreviation $\na_z=\na_{\partial_z}$, we may certainly write $\na_z\partial_z=\Gamma_1\partial_z+\Gamma_2\partial_\cz$ and, due to vanishing torsion, write $\na_\cz\partial_z=\na_z\partial_\cz=\Gamma_3\partial_z+\Gamma_4\partial_\cz$, for some complex-valued functions $\Gamma_1,\ldots,\Gamma_4$ on $\calu$. Since $\na$ is a real operator, $\na_\cz\partial_\cz=\overline{\Gamma}_2\partial_z+\overline{\Gamma}_1\partial_\cz$ as well as $\Gamma_4=\overline{\Gamma}_3$. This leads to the following Proposition.
\begin{prop}\label{prop_somebasederivatives}
The Levi-Civita connection of $g$ is given by
\begin{equation}
 \na_z\dx z=-\Gamma\dx z, \qquad 
 \na_z\dx\cz=0 ,
\end{equation}
\begin{equation}\label{Gammasfinalmente}
 \Gamma:=\Gamma_1=\frac{1}{\lambda}\papa{\lambda}{z},\qquad \Gamma_2=\Gamma_3=\Gamma_4=0.
\end{equation}
\end{prop}
Let us also recall the formulae for the curvature and sectional curvature of $M$ are given respectively by
\begin{equation}\label{curvatura1}
 R^\na(\partial_z,\partial_\cz)\partial_z  = \na_z\na_\cz\partial_z-\na_\cz\na_z\partial_z=-\papa{\Gamma}{\cz}\partial_z=-\papa{^2\log\lambda}{z\partial\cz}\partial_z ,
\end{equation}
\begin{equation}\label{curvatura2}
 K=\frac{g(R^\na(\partial_z,\partial_\cz)\partial_z,\partial_\cz)}{g(\partial_z,\partial_\cz)^2}=-\frac{2}{\lambda}\papa{\Gamma}{\cz}=-\frac{2}{\lambda}\papa{^2\log\lambda}{z\partial\cz}.
\end{equation}

Next we consider the open subset $T_\calu=\inv{\pi}(\calu)\subset T_M$. We have trivialising coordinates
\begin{equation}
 T_\calu=\{ (z,w):\ z\in\calu,\ w\in\C\}
\end{equation}
and hence, writing $w=s+it,\ s,t\in\R$, we have the tautological vector field
\begin{equation}\label{thetautologicalvectorfieldincoordinates}
  U_{(z,w)}=s\partial_s+t\partial_t=w\partial_w+\cw\partial_\cw.
\end{equation}

Clearly, isothermal coordinates $z$ on $M$ are compatible with an induced integrable complex structure, which is unique up to orientation if $M$ is orientable. Transition maps between isothermal coordinates are indeed holomorphic, as it is easy to prove. This implies that the tangent bundle $TM\lrr M$ inherits a complex structure from $M$ and the structure of a holomorphic vector bundle. In particular $T_M$ becomes a holomorphic manifold, cf. \cite{Koba1}.

It follows that $\partial_w$ generates the vertical $+i$-eigenvectors or $(1,0)$-vectors in $(TT_\calu)^\C$. Now we recall the pullback connection $\na^*$ and search for a generator of the horizontal $+i$-eigenbundle.
\begin{prop}\label{Prop_Sasakimetrichorizontalonezerovectorfield}
 We have $({H^{\na}_{(z,w)}})^{1,0}=\C X$ where
 \begin{equation}
   X=\partial_z-w\Gamma\partial_w .
 \end{equation}
\end{prop}
\begin{proof}
  We may assume $X=\partial_z+\epsilon\partial_w$, for some $\epsilon\in\C$, is a generator of $({H^{\na}})^{1,0}$. Thus satisfying $\na^*_XU=0$. By construction, we have $B(\pi^*\partial_z)=\partial_w$. And therefore
 \begin{eqnarray*}
   \na^*_X(w\partial_w+\cw\partial_\cw)&=& \epsilon\partial_w+wB\na^*_X\pi^*\partial_z+\cw B\na^*_X\pi^*\partial_\cz \\
      &=& \epsilon\partial_w+wB(\Gamma\pi^*\partial_z ) \\
      &=& (\epsilon+w\Gamma)\partial_w 
 \end{eqnarray*}
 yields the result.
\end{proof}
In particular we see $X=\pi^*\partial_z$ is the horizontal lift of $\partial_z$. Of course, $\partial_w$ is the vertical.
Next we find a $(1,0)$-form over $T_\calu$ such that $\eta(\partial_w)=1,\ \eta(X)=0$. Clearly\footnote{We have the following result concerning transition functions.
\begin{propsemnumero}
Let $(\calu_1,z_1)$ denote another complex chart of $M$ defined on an open subset $\calu_1$ with non-empty intersection with $\calu$. Consider the corresponding chart $(z_1,w_1)$ of $T_M$. Then:
 \[ \papa{}{z_1}=\papa{z}{z_1}\papa{}{z}, \quad w_1=\papa{z_1}{z}w, \quad \lambda_1=\biggl|\papa{z}{z_1}\biggr|^2\lambda,
 \quad \Gamma_1=\papa{z}{z_1}\Gamma+\papa{z_1}{z}\papa{^2z}{{z_1}^2}, \quad \eta_1=\papa{z_1}{z}\eta . \]
 \end{propsemnumero}
 Certainly the 1-form $\eta$ is a covariant tensor and thus transforms like a tensor. However, $\eta$ is defined on two complex dimensions. The third and the last identity hence yield the main result of Section \ref{sec:Thenewmetric}, i.e. that $\lambda\,\dx z\ceta$ is a well-defined global tensor.},
\begin{equation}\label{theetaform}
 \eta=w\Gamma\dx z+\dx w.
\end{equation}
We remark that $\dx\eta^{(0,2)}=0$ yields the holomorphic structure of the manifold $T_M$. 

\begin{prop}
 The Sasaki metric $\ciconiametric_{1,0,1}$ over $T_\calu=\calu\times\C$ coincides with $\lambda(\dx z\dx\cz+\eta\ceta)$. More precisely, $\pi^*g=\lambda\,\dx z\dx\cz$ and $\pi^\estrela g=\lambda\,\eta\ceta$.
\end{prop}
Next we compute a few derivatives for later purposes. First, as above and by trivial reasons, we have
\begin{equation}
  \na^*_z\partial_w=\Gamma\partial_w,\qquad
 \na^*_z\partial_\cw=\na^*_\cz\partial_w=0 ,
 \qquad \na^*_w\partial_w= \na^*_\cw\partial_w=0 .
\end{equation}
Then
\begin{equation}
 \begin{split}
 \na^*_z\partial_z &=\: \na^*_z(\pi^*\partial_z+w\Gamma\partial_w) \\
    &=\: \Gamma\pi^*\partial_z+w\papa{\Gamma}{z}\partial_w+w\Gamma^2\partial_w \\
    &=\: \Gamma\partial_z+w\papa{\Gamma}{z}\partial_w  
 \end{split}
\end{equation}
and, in the same way,
\begin{equation}
   \na^*_z\partial_\cz\: =\: \cw\papa{\overline{\Gamma}}{z}\partial_\cw, \qquad
   \na^*_\cz\partial_z\:=\:w\papa{\Gamma}{\cz}\partial_w.
\end{equation}
Easy enough,
\begin{equation}
 \na^*_w\partial_z=\Gamma\partial_w, \qquad  \na^*_\cw\partial_z=0.
\end{equation}
\begin{prop}\label{prop_somederivatives}
We have
\begin{equation}            \label{somederivatives}
 \na^*_z\begin{cases}
         \dx z \\ \dx\cz\\ \dx w\\ \dx \cw  \end{cases}=\ \begin{cases}-\Gamma\dx z \\ 0 \\-w\papa{\Gamma}{z}\dx z-\Gamma\dx w \\-\cw\papa{\overline{\Gamma}}{z}\dx\cz \end{cases}
  \qquad\qquad
   \na^*_w\begin{cases}
         \dx z \\ \dx\cz\\ \dx w\\ \dx \cw  \end{cases}=\ \begin{cases} 0 \\ 0 \\-\Gamma\dx z \\ 0 \end{cases} .         
\end{equation}
\end{prop}
Checking $\na^*(\lambda\,\dx z\dx\cz)=\na^*\pi^*g=0$, as expected, is now a simple exercise.

Regarding the vertical part, we start by
\begin{equation}  \label{somemorederivatives1}
 \na^*_z\eta=\na^*_z(w\Gamma\dx z+\dx w)=w\papa{\Gamma}{z}\dx z-w\Gamma^2\dx z-w\papa{\Gamma}{z}\dx z-\Gamma\dx w=-\Gamma\eta .
\end{equation}
Further we find
\begin{equation} \label{somemorederivatives2}
\begin{split}
 & \na^*_\cz\eta=\na^*_\cz(w\Gamma\dx z+\dx w)=w\papa{\Gamma}{\cz}\dx z-w\papa{\Gamma}{\cz}\dx z=0, \\
 &\na^*_w\eta=\Gamma\dx z-\Gamma\dx z=0,\qquad\qquad \na^*_\cw\eta=\na^*_w\ceta=0.
\end{split}
\end{equation}
And thus the identity $\na^*(\lambda\,\eta\ceta)=\na^*\pi^\estrela g=0$ follows, again just as expected.

We remark that the torsion of the metric connection $\na^*$ is proportional to $\pi^*R^\na U$.

The ciconia metric with weights $f,a,h\in\cinf{}(T_M;\C)$, where $f,h$ are real, may now be introduced in coordinates:
\begin{equation}\label{Definitionmetric_G_incomplexcoord}
 \ciconiametric_{f,a,h}=\frac{\lambda}{2}(f\,\dx z\cdot\dx\cz
             +a\,\dx z\cdot\ceta+\ca\,\eta\cdot\dx\cz+h\,\eta\cdot\ceta ) .
\end{equation}
We use momentarily $\alpha\cdot\beta=\alpha\otimes\beta+\beta\otimes\alpha$, so that for instance $g=\frac{\lambda}{2}\,\dx z\cdot\dx\cz$. Notice the middle term $g_a=\frac{\lambda}{2}(a\,\dx z\cdot\ceta+\ca\,\dx\cz\cdot\eta)$ is also parallel for $\na^*$ when $a$ is constant: indeed, by (\ref{somederivatives}--\ref{somemorederivatives2}),
\begin{equation}
 \na^*(\lambda\,\dx z\cdot\ceta)=0 \qquad \mbox{and}\qquad \na^*(\lambda\,\eta\cdot\dx\cz)=0.
\end{equation}
This is just as predicted by the theory: the invariance of $g_a$ under the diagonal representation of $\SO(2)=\Uni(1)$ and the reduction of $\na^*$ as a linear connection on $T_M$, as mentioned, in general with torsion. As observed in Section \ref{sec:Introduction}, every ciconia metric $\ciconiametric_{f,a,h}$ defined on $T_M$ is compatible with the underlying  holomorphic structure.

%% file: symplecticiconia.tex
\section{First developments}

\subsection{K\"ahlerian and pseudo-K\"ahlerian ciconia metrics}

\label{sec:Km}

Let us recall that a pseudo-Hermitian structure or metric on a complex manifold is given by a $\C$-linear map (the complex structure on $T^{0,1}$ is multiplication by $-i$, so that $-i\overline{u}=\overline{iu}$)
\begin{equation}
 H:T^{1,0}\otimes T^{0,1}\lrr \C \ \ \mbox{such that}\ \ H(\overline{v},\overline{u})=\overline{H(u, v)},\ \forall u\in T^{1,0},v\in T^{0,1} .
\end{equation}
On a chart $(z^j)_{j=1,\ldots,\dim M}$, a pseudo-Hermitian metric appears as $\sum h_{jk}\dx z^j\otimes\dx\cz^k$ with $h_{jk}=\overline{h}_{kj}$. Then $H$ is compatible with the complex structure and the same is true for the real and imaginary parts of $H$. The real part $\Re H$ is the associated pseudo-Riemannian metric and minus the imaginary part $-\Im H=\Re( iH)$ is the associated symplectic 2-form.

Resuming with the analysis of a given ciconia metric on the complex manifold $T_M$, let us consider the pseudo-Hermitian structure $H_a$ given locally by
\begin{equation}
 H_a=\lambda a\,\dx z\otimes\ceta+\lambda\ca\,\eta\otimes\dx\cz.
\end{equation}
$H_a$ is the pseudo-Hermitian structure associated to $g_a$ introduced in \eqref{Definitionmetric_g_a}. Indeed, $g_{0,a,0}$ from  \eqref{Definitionmetric_G_incomplexcoord} agrees with $g_a=\Re H_a=(H_a+\overline{H}_a)/2$. Now we have the symplectic form $g_a(i\ ,\ )$
\begin{equation}
\begin{split}
  -\Im H_a &= \frac{i\lambda}{2}(a\,\dx z\otimes\ceta+\ca\,\eta\otimes\dx\cz
              -\ca\,\dx \cz\otimes\eta-a\,\ceta\otimes\dx z) \\
     &= \frac{i\lambda}{2}(a\,\dx z\wedge\ceta-\ca\,\dx \cz\wedge\eta) .
\end{split}
\end{equation}

Finally we have the symplectic 2-form of a ciconia metric $\ciconiametric_{f,a,h}$:
\begin{equation}
 \omega_{f,a,h}=\frac{i\lambda}{2}(f\,\dx z\wedge\dx\cz+a\,\dx z\wedge\ceta+\ca\,\eta\wedge\dx\cz+h\,\eta\wedge\ceta).
\end{equation}
\begin{prop}\label{prop_omegaclosedequation}
 $\omega_{f,a,h}$ is closed if and only if 
 \begin{equation} \label{omegaclosedequation}
  \begin{cases}
    \dpapa{h}{z}=\dpapa{a}{w}+w\Gamma\dpapa{h}{w}  \vspace*{4mm} \\
   \dpapa{f}{w}+w\Gamma\dpapa{\ca}{w}=\dpapa{\ca}{z}+\cw h\dpapa{\Gamma}{\cz} .
  \end{cases}
 \end{equation}
\end{prop}
\begin{proof}
 Since $\eta=w\Gamma\dx z+\dx w,\ \ceta=\cw\cgamma\dx\cz+\dx\cw$, we find
 \begin{eqnarray*}
  \eta\wedge\ceta &=& |w|^2|\Gamma|^2\dx z\wedge\dx\cz+w\Gamma\dx z\wedge\dx\cw-\cw\cgamma\dx\cz\wedge\dx w+\dx w\wedge\dx\cw,\\
  \dx\eta &=& \Gamma\dx w\wedge\dx z-w\papa{\Gamma}{\cz}\dx z\wedge\dx\cz, \\
  \dx\ceta &=& \cgamma\dx\cw\wedge\dx\cz+\cw\papa{\cgamma}{z}\dx z\wedge\dx\cz .
 \end{eqnarray*}
 Then
 \begin{eqnarray*}
  -2i\omega_{f,a,h} &=& (f\lambda+h\lambda|w|^2|\Gamma|^2+a\lambda\cw\cgamma+\ca\lambda w\Gamma)\dx z\wedge\dx\cz+(a\lambda+h\lambda w\Gamma)\dx z\wedge\dx\cw \\
  & &\ \ +(-\ca\lambda-h\lambda\cw\cgamma)\dx\cz\wedge\dx w+h\lambda\dx w\wedge\dx\cw  \\
 &=& \lambda(f+hw\cw\Gamma\cgamma+a\cw\cgamma+\ca w\Gamma)\dx z\wedge\dx\cz+(a\lambda+h w\papa{\lambda}{z})\dx z\wedge\dx\cw \\
  & &\ \ -(\ca\lambda+h\cw\papa{\lambda}{\cz})\dx\cz\wedge\dx w+h\lambda\dx w\wedge\dx\cw . \end{eqnarray*}
 Now, recalling $\lambda$ is positive and depends only of $z$,
 \begin{eqnarray*}
 \lefteqn{ -2i\dx\omega_{f,a,h} \,= } \\
 &=& \lambda(\papa{f}{w}+\papa{h}{w} w\cw\Gamma\cgamma+h\cw\Gamma\cgamma+\papa{a}{w}\cw\cgamma+\papa{\ca}{w}w\Gamma+\ca\Gamma)\dx z\wedge\dx\cz\wedge\dx w+ \\
  & & + \lambda(\papa{f}{\cw}+\papa{h}{\cw} w\cw\Gamma\cgamma+hw\Gamma\cgamma+\papa{a}{\cw}\cw\cgamma+a\cgamma+\papa{\ca}{\cw}w\Gamma)\dx z\wedge\dx\cz\wedge\dx\cw \\
  & & -(\papa{a}{\cz}\lambda+a\papa{\lambda}{\cz}+w\papa{h}{\cz}\papa{\lambda}{z}+hw\papa{^2\lambda}{\cz\partial z})\dx z\wedge\dx\cz\wedge\dx\cw \\
  & & -(\papa{a}{w}\lambda+w\papa{h}{w}\papa{\lambda}{z}+h\papa{\lambda}{z})\dx z\wedge\dx w\wedge\dx\cw \\ 
  & & -(\papa{\ca}{z}\lambda+\ca\papa{\lambda}{z}+\cw\papa{h}{z}\papa{\lambda}{\cz}+h\cw\papa{^2\lambda}{z\partial\cz})\dx z\wedge\dx\cz\wedge\dx w \\
  & & -(\papa{\ca}{\cw}\lambda+\cw\papa{h}{\cw}\papa{\lambda}{\cz}+h\papa{\lambda}{\cz})\dx\cz\wedge\dx w\wedge\dx\cw \\ 
  & & +(\papa{h}{z}\lambda+h\papa{\lambda}{z})\dx z\wedge\dx w\wedge\dx\cw +(\papa{h}{\cz}\lambda+h\papa{\lambda}{\cz})\dx\cz\wedge\dx w\wedge\dx\cw .
 \end{eqnarray*}
 Thus $\omega_{f,a,h}$ is closed if and only if the following system is satisfied:
 \begin{equation*}
  \begin{cases}
   \papa{h}{z}\lambda-\papa{a}{w}\lambda-w\papa{h}{w}\papa{\lambda}{z}=0  \vspace{1mm}\\
   \lambda\bigl(\papa{f}{w}+w\cw\Gamma\cgamma\papa{h}{w}+\cw\Gamma\cgamma h +\cw\cgamma\papa{a}{w}+w\Gamma\papa{\ca}{w}+\ca\Gamma\bigr)   = \vspace{1mm} \\
   \hspace{45mm} =\papa{\ca}{z}\lambda+\ca\papa{\lambda}{z}+\cw\papa{h}{z}\papa{\lambda}{\cz}+h\cw\papa{^2\lambda}{z\partial\cz} .
  \end{cases}
 \end{equation*}
 Equivalently, 
 \begin{equation*}
  \begin{cases}
   \papa{h}{z}=\papa{a}{w}+w\Gamma\papa{h}{w} \vspace{1mm}\\
   \papa{f}{w}+\cw\Gamma\cgamma h+\cw\cgamma\papa{h}{z}+w\Gamma\papa{\ca}{w}+\ca\Gamma
   =\papa{\ca}{z}+\ca\Gamma+\cw\papa{h}{z}\cgamma+ h\cw\papa{^2\lambda}{z\partial\cz}\frac{1}{\lambda} .
  \end{cases}
 \end{equation*}
 Since
 \[  \papa{\Gamma}{\cz}=\papa{ }{\cz}(\papa{\lambda}{z}\frac{1}{\lambda})=\papa{^2\lambda}{z\partial\cz}\frac{1}{\lambda}-\Gamma\cgamma , \]
 another substitution in the previous equation yields the result.
\end{proof}


$T_M$ is always a complex analytic manifold associated to the Riemann surface $M$, endowed with a smooth Riemannian structure. With respect to the Gray-Hervella classification of Hermitian 4-manifolds, one may only distinguish further the metrics which are K\"ahler or pseudo-K\"ahler. It is well-know the condition for these is the same: $\dx\omega_{f,a,h}=0$.

Two sets of $\C$-valued functions on $T_M$ are worth considering in order to reduce the indeterminacy of \eqref{omegaclosedequation}. The first, $\cinf{\calu,\pi}$, is the set of functions which are the pullback by $\pi$ of functions on $M$, i.e. functions which depend only of $z$. The second set, denoted $\cinf{r^2}=\cinf{[0,+\infty[}(\C)$, where $r^2=r^2(u)=g(u,u)=\lambda|w|^2,\ u\in T_M$, is the set of functions $\varphi$ on $\inv{\pi}(\calu)$ which depend only of $r^2$ and have derivatives $\varphi',\varphi'',\ldots$ at 0 (n.b.: we let $\varphi'=\dx\varphi/\dx r^2$).

Let us find first which ciconia metrics are K\"ahler, separately from the pseudo-K\"ahler. By Proposition \ref{prop_signatureofciconiametric}, this means that $f,h$ are real and  $f,fh-|a|^2>0$. Recall from \eqref{curvatura2} the notation by $K$ for the Gauss curvature of $(M,g)$.
\begin{teo}\label{teo_Kahlerciconia}
 Suppose a given ciconia metric $\ciconiametric_{f,a,h}$ is K\"ahler with weight functions of any of the two types above. We have that:\\
 (i) if $f,a,h\in\cinf{M,\pi}$, then $K=0$, $a$ is holomorphic and $h$ is constant;\\
 (ii) if $f,h\in\cinf{M,\pi}$ and $a\in\cinf{r^2}$, then $K=0$ and $a,h$ are constant;\\
 (iii) if $f,a\in\cinf{M,\pi}$ and $h\in\cinf{r^2}$, then $K=0$ and $a$ is holomorphic;\\
 (iv) if $a,h\in\cinf{M,\pi}$ and $f\in\cinf{r^2}$, then $f(r^2)=f_1r^2+f_0$, $K=-\frac{2f_1}{h}$, $h,f_0,f_1$ are constant and $a$ is holomorphic;\\
  (v) if $f\in\cinf{M,\pi}$ and $a,h\in\cinf{r^2}$, then $K=0$ and $a$ is constant;\\
 (vi) if $a\in\cinf{M,\pi}$ and $f,h\in\cinf{r^2}$, then $K=-\frac{2f'}{h}$ and $a$ is holomorphic;\\
 (vii) if $h\in\cinf{M,\pi}$ and $f,a\in\cinf{r^2}$, then  $f(r^2)=f_1r^2+f_0$, $K=-\frac{2f_1}{h}$, $a,h,f_0,f_1$ are constant;\\
 (viii) if $f,a,h\in\cinf{r^2}$, then $K=-\frac{2f'}{h}$ and $a$ is constant.
 
 Reciprocally, any of the conditions above imply the metric is K\"ahler.
\end{teo}
\begin{proof}
We give some starting details of the proof and leave the others to the reader.
Since $r^2=\pi^\estrela g(U,U)$, where $U$ is the tautological section, applying $\na^*$ we find $\dx r^2(x)=2\pi^\estrela g(U,\na^*_xU)=\lambda(\eta\otimes\ceta+\ceta\otimes\eta)(w\partial_w+\cw\partial_\cw,x^v)=\lambda (w\ceta+\cw\eta)(x)$, $\forall x\in T(T_M)$; this is just a coherent, albeit complicated, way of proving a formula for $\dx r^2$ where $r^2=\lambda w\cw$. Immediately, any function $\varphi\in\cinf{r^2}$ satisfies
\[ \papa{\varphi}{w}= \varphi'\lambda\cw, 
          \qquad\papa{\varphi}{z}=\varphi'\papa{\lambda}{z}w\cw . \]
These imply $\papa{\varphi}{z}=w\Gamma\papa{\varphi}{w}$. Using this either on $\overline{a}$, on $h$, or on both, the proof continues with the analysis of the system for $\dx\omega_{f,a,h}=0$ in Proposition \ref{prop_omegaclosedequation}. 

 We show case (vi) since this is the most exemplary of all $2^3$ cases and since it is the hardest to see. We have $\papa{a}{w}=\papa{a}{\cw}=0$. For $f$ or $h$ functions of $r^2=\lambda|w|^2$ we have the seen identities
\[ \papa{f}{w}=f'\lambda\cw,\qquad\papa{f}{z}=f'\papa{\lambda}{z}|w|^2\]
 which imply $\papa{f}{z}=w\Gamma\papa{f}{w}$. 
 
 System \eqref{omegaclosedequation} hence becomes equivalent to a single equation:
 \[  f'\lambda \cw=\papa{\ca}{z}+\cw h\papa{\Gamma}{\cz}  .  \qquad\qquad (*) \]
 Then differentiating this equation with respect to $w$ yields $f''\lambda^2\cw^2=\cw^2h'\lambda\papa{\Gamma}{\cz}$, $\forall w\in\C$, i.e. $f''\lambda=h'\papa{\Gamma}{\cz}$. 
 Taking the same equation (*) and differentiating with respect to $\cw$ yields
 \[ f''\lambda^2|w|^2+f'\lambda=h\papa{\Gamma}{\cz}+|w|^2\lambda h'\papa{\Gamma}{\cz}  .\] 
 With the input of the previous identity, we find $f'\lambda=h\papa{\Gamma}{\cz}$. This is the desired equation $-2f'=Kh$. Looking up again on (*) we find the desired result $\papa{a}{\cz}=0$. The previous conclusions also follow immediately from $w=0$ in (*), but we wish to avoid such argument.
\end{proof}

The result is indeed global, although the system of PDE is local. Notice we end up with the following two inclusion diagrams
\begin{equation}
 \mbox{(ii)}\Rightarrow\mbox{(i)}\Rightarrow\mbox{(iii)}\Leftarrow\mbox{(v)}\qquad\qquad
 \mbox{(vii)}\Rightarrow\mbox{(iv)}\Rightarrow\mbox{(vi)}\Leftarrow\mbox{(viii)} .
\end{equation}
In sum, each and every of the above cases is described by (iii) or (vi).

We observe, as supplement of case (iv) of the Theorem, that if $K>0$, then we have $f_1<0,\ f_0>0,\ r^2\in[0,-\frac{f_0}{f_1}[$ and so the ciconia metric is defined only on a disk-bundle and is non-complete (on the fibres the metric is Euclidean because $h$ is constant); if $K=0$, then $f=f_0$ and $a$ is bounded (hence constant if $M$ is compact); finally, if $K<0$, then we deduce $f_1>0$. Now we also observe case (iii) returns to case (i) once we ask that $fh-|a|^2$ should be a constant: because $h$ cannot vary along the fibres. Hence the following example is surprisingly interesting here.

\begin{prop}\label{prop_nice_ciconia_metric}
For any open set $\calu\subset\C$ and any $a\in\cinf{\calu}$ holomorphic in $z$, the Hermitian metric on $\calu\times\C$
\begin{equation}
H=(1+|a|^2)\dx z\otimes\dx\cz+a\dx z\otimes\dx\cw+\ca\dx w\otimes\dx\cz+\dx w\otimes\dx\cw 
\end{equation}
is K\"ahler and flat.
\end{prop}
\begin{proof}
 We give two proofs. For the first, we just notice the biholomorphism
 \[ F(z,w)=(z,w+\int a(z)\dx z) \]
 from $T_\calu=\calu\times\C$ with the given metric $H$ onto itself with the canonical metric. Finding such isometric biholomorphism just conforms with the classification of Hermitian symmetric domains. A second proof goes as follows. Recalling the unique associated Hermitian connection on $TT_\calu\rr T_\calu$ whose $(0,1)$-part coincides with $\db$ is given by $\tilde{\omega}^{\mathrm{t}}=(\partial H)\inv{H}$, where $H$ is the Hermitian metric matrix on a holomorphic frame, cf. \cite[Chap. I, \S4)]{Koba1}, we then have
 \begin{equation}
     \tilde{\omega}^\mathrm{t}= \left[\begin{array}{cc}
        \ca\papa{a}{z}\dx z &  \papa{a}{z}\dx z \\ 0 & 0 \end{array} \right]
        \left[ \begin{array}{cc}   1 &-a  \\ -\ca & 1+|a|^2    \end{array} \right]
      =\left[\begin{array}{cc}  0 &\papa{a}{z}\dx z \\ 0 &0  
      \end{array}\right] .
 \end{equation}
 The connection is torsion-free when the metric is K\"ahler, which is the case as deduced before and can be immediately checked. In other words, the Hermitian metric gives the Levi-Civita connection. Since $\db\tilde{\omega}$ is the curvature (1,1)-form, the result follows.
\end{proof}

We may construct the following singular spaces. Let us take the Weierstrass $\wp$-function in the $z$ complex plane with period lattice $\Lambda_1$ and let $\Lambda_2$ be another fixed lattice in the $w$ plane. Then the metric from Proposition \ref{prop_nice_ciconia_metric}, with $a(z)=\wp(z)$, descends partially to the toric manifold $\C/\Lambda_1\times\C/\Lambda_2$. Unfortunately it is not defined on the $\Lambda_1\times\C/\Lambda_2$, since $\Lambda_1$ is the pole set of $\wp$. 


The next result is true for trivial reasons. The purpose is to notice the possibilities of case (vi) in Theorem \ref{teo_Kahlerciconia} are stricter than it is shown.
\begin{prop}\label{prop_baseKconstant}
 In all cases in Theorem \ref{teo_Kahlerciconia}, the curvature $K$ is constant.
\end{prop}
\begin{proof}
For case (vi) and (viii), since $K$ is independent of $r^2$, we have $Kh'+2f''=0$. Then 
 \[ 0=\papa{(Kh+2f')}{z}=\papa{K}{z}h+Kh'\papa{\lambda}{z}|w|^2+2f''\papa{\lambda}{z}|w|^2 =\papa{K}{z}h .\]
 Since $h>0$, the result follows.
\end{proof}

Ciconia metrics which are pseudo-K\"ahler, i.e. with holonomy in $\Uni(1,1)$, may be found via the same restricting hypothesis as above. The weight functions now satisfy merely $fh-|a|^2<0$.
\begin{teo}
 A ciconia metric $\ciconiametric_{f,a,h}$ with weight functions of the two types above is pseudo-K\"ahler if $f,a,h$ satisfy any of the conditions (i---viii) in Theorem \ref{teo_Kahlerciconia}, with $h\neq0$, or any of the following when $h=0$ identically: \\
 (ix) if $a\in\cinf{M,\pi}$, then $f\in\cinf{M,\pi}$ and $a$ is holomorphic;\\
 (x) if $a\in\cinf{r^2}$, then $f,a$ are constant.
\end{teo}
The proof is immediate. There is more choice for $K$ on pseudo-K\"ahler metrics, as we shall see next.
\begin{coro}
 Let $f,h$ be constants and let $a\in\cinf{M,\pi}\cup\cinf{r^2}$. \\
 (i) The ciconia metric $\ciconiametric_{f,a,h}$ is K\"ahler if and only if $f,fh-|a|^2>0$, $a\in\cinf{M,\pi}$ is holomorphic and $K=0$.\\
 (ii) For any $a\in\cinf{M,\pi}$ non-vanishing and holomorphic, $\ciconiametric_{f,a,0}$ is pseudo-K\"ahler.
\end{coro}
Clearly, in most situations of case (i), we end up with $a$ constant. We
recall the metric $\ciconiametric_{0,1,0}$ was known to K. Yano. The special ciconia metrics introduced here deserve a more detailed study, regarding the questions of geodesics, completeness, curvatures, complex and Lagrangian submanifolds, which cannot be pursued here. 

A last important remark regards the case of previously seen ciconia metrics. Suppose $\phi:M\lrr M$ is an isometry of the base $(M,g)$. It is proved in \cite{Alb2014d} that, given any pair of functions $f,h\in\cinf{r^2}$, then the induced map $\phi_*:T_M\lrr T_M$ is an  isometry for the metric $\ciconiametric_{f,0,h}$. Supposing now, moreover, that we have a function $a\in\cinf{M,\pi}$, then we find that $\phi_*$ is an isometry for the ciconia metric $\ciconiametric_{f,a,h}$ if and only if 
\begin{equation}
 \dx\phi_z\circ a_z=a_{\phi(z)}\circ\dx\phi_z.
\end{equation}
Indeed, $a$ may be seen as a vector bundle morphism along $\phi$. The proof is easy, combining the result in \cite[Theorem 1.3]{Alb2014d} with the definition \eqref{Definitionmetric_g_a}.

\subsection{K\"ahlerian Ricci-flat ciconia metric}

Here we assume the given metric $\ciconiametric_{f,a,h}$ on $T_M$ is positive definite. Recalling the notation $\fhmenosaquadrado=fh-|a|^2$, we thus have $f,\fhmenosaquadrado>0$. Since $T_M$ is a Hermitian manifold with Hermitian metric
\begin{equation}
 H = \lambda\bigl(f\,\dx z\otimes\dx\cz+a\,\dx z\otimes\ceta + \ca\,\eta\otimes\dx\cz+h\,\eta\otimes\ceta \bigr)
\end{equation}
we may use the well-known formula for the Ricci-form, the closed $(1,1)$-form
\begin{equation}
 \rho=i\db\partial\log\det H .
\end{equation}
\begin{prop}
 We have
 \begin{equation}\label{formula_rho}
  \rho=2K\pi^*\omega+i\db\partial\log \fhmenosaquadrado .
 \end{equation}
\end{prop}
\begin{proof}
 By invariance of the unitary structure we may apply the type $(1,0)$ frame field $\pi^*\partial_z,\partial_w$ from Proposition \ref{Prop_Sasakimetrichorizontalonezerovectorfield}, and therefore deduce that $\det H=\lambda^2\fhmenosaquadrado$. Combining with \eqref{curvatura2} and $\pi^*\omega=\frac{i\lambda}{2}\dx z\wedge\dx\cz$, the result follows. We notice that
 \begin{align*}
  H &=\: \lambda\bigl(f+h|w|^2|\Gamma|^2+2\Re(\ca w\Gamma)\bigr)\,\dx z\otimes\dx\cz+\lambda(a+hw\Gamma)\,\dx z\otimes\dx\cw +  \\
  &  \qquad\qquad  +\lambda(\ca+h\cw\cgamma)\,\dx w\otimes\dx\cz+\lambda h\,\dx w\otimes\dx\cw 
 \end{align*}
 so $\det H$ is also confirmed from the matrix of $H$ on the holomorphic frame $\partial_z,\partial_w$.
\end{proof}
Following the well-know theory of Hermitian manifolds, we say that $T_M$ is K\"ahler-Einstein if $\rho=\frac{S}{4}\,\omega_{f,a,h}$, for some $S\in\R\backslash0$. For obvious reason, the space is said to be Ricci-flat if $\rho=0$. Returning to the formula in the proof of Proposition \ref{prop_omegaclosedequation}, we deduce the next result.
\begin{prop}
 The equations for an Einstein metric $\ciconiametric_{f,a,h}$, with Einstein constant $S/4$, are equivalent to the following local system:
 \begin{equation}\label{Einsteinciconiametric}
 \begin{cases}
  \dpapa{^2\log \fhmenosaquadrado}{w\partial\cw}=-\dfrac{\lambda S h}{8}, \vspace*{4mm}\\
  \dpapa{^2\log \fhmenosaquadrado}{z\partial\cw}=-\dfrac{\lambda S}{8}(a+hw\Gamma), \vspace*{4mm} \\
  {\lambda}K-\dpapa{^2\log \fhmenosaquadrado}{z\partial\cz}=\dfrac{\lambda S}{8}(f+\ca w\Gamma +a\cw\cgamma+h|w|^2|\Gamma|^2) .  
 \end{cases}
 \end{equation}
\end{prop}
Let $T_M\backslash M$ denote the complement of the zero-section.
\begin{teo}\label{teo_Ricciflatopen}
 For any Riemann surface $(M,g)$ and every set of smooth functions $f,a,h$ on $T_M$ such that $f>0$ and $fh-|a|^2=\frac{1}{r^4}$, the ciconia metric $\ciconiametric_{f,a,h}$ on the open manifold $T_M\backslash M$ is Ricci-flat.
\end{teo}
\begin{proof}
 Since $r^2=\lambda w\cw$, then $\db\partial\log r^{-4}=-2(\partial^2_{\cz,z}\log\lambda)\dx \cz\wedge\dx z=\lambda K\,\dx\cz\wedge\dx z$. Substituting this result in \eqref{formula_rho}, yields $\rho=0$.
\end{proof}

In fact the solution $1/r^4$ is found when we factor out $\fhmenosaquadrado$ by a function of $r^2$. It is unique, in this way, for the Ricci-flat ciconia metric. Hence we may always write $\fhmenosaquadrado=\psi/{r^4}$, out of the zero-section, and study the system \eqref{Einsteinciconiametric} with $0$ in the place of $K$ and $\psi$ in the place of $\fhmenosaquadrado$.

The next result is an improvement of the previous, meeting with the well-known problem of finding examples of complete K\"ahler Ricci-flat manifolds. We construct \textit{almost} Calabi-Yau spaces, giving evidence to the irrevocable role of the weight-function $a$.
\begin{teo}
 Let $(M,g)$ be an oriented compact Riemann surface of constant Gauss curvature $K=-1,0$ or $1$. For any $0\leq\epsilon_<\epsilon_2$, let us denote by $Z=Z_{\epsilon_1,\epsilon_2}$ the  submanifold $Z=\{u\in T_M:\ \epsilon_1<r^2<\epsilon_2\}$, open in $T_M$, where $r^2=g(u,u)$. Given the following conditions on a constant $c_0\in\R$ and on $\epsilon_1,\epsilon_2,f,a,h$, the respective ciconia metrics $\ciconiametric_{f,a,h}$ on the manifold $Z$ are K\"ahler and Ricci-flat:\\
  (i) if $K=0$, we consider $Z_{0,+\infty}$ with
  \begin{equation}
   f(r^2)=f>0\ \mbox{constant},\qquad a\neq0\ \mbox{constant},\qquad h(r^2)=\frac{|a|^2}{f}+\frac{1}{fr^4}
  \end{equation}
  and then the metric is complete.\\
  (ii) if $K=1$, we let $Z=Z_{0,{\frac{1}{c_0}}}$, for any $c_0>0$, and take
  \begin{equation}
   f(r^2)=\frac{\sqrt{1-c_0r^2}}{r},\qquad a=0,\qquad h(r^2)=\frac{1}{r^3\sqrt{1-c_0r^2}};
  \end{equation}
   and then the associated metric space may be completed to $\overline{Z}\backslash M$.\\
  (iii) also if $K=1$, we let $Z=Z_{0,\beta_+}$, where $\forall c_0\in\R$
 \begin{equation}
  \beta_+=\frac{-c_0+\sqrt{c_0^2+4|a|^2}}{2|a|^2},
 \end{equation}
 and let
 \begin{equation}
  f=\frac{1}{r}\sqrt{-|a|^2r^4-c_0r^2+1},\quad a\neq0\ \mbox{constant},\quad
  h=\frac{|a|^2r^4+1}{r^3\sqrt{-|a|^2r^4-c_0r^2+1}},
 \end{equation}
 so that the associated metric space structure on $\overline{Z}\backslash M$ is complete.\\
 (iv) if $K=-1$, we let $Z=Z_{\beta_+,+\infty}$ with $\beta_+$ as above
 and take
 \begin{equation}
  f=\frac{1}{r}\sqrt{|a|^2r^4+c_0r^2-1},\quad a\neq0\ \mbox{constant},\quad
  h=\frac{|a|^2r^4+1}{r^3\sqrt{|a|^2r^4+c_0r^2-1}},
 \end{equation}
 implying the associated metric space structure on $\overline{Z}$ is complete.
 
 Reciprocally, the above are all the Cauchy-complete solutions arising from the conditions found in Theorems \ref{teo_Kahlerciconia} and \ref{teo_Ricciflatopen}, up to conformal change.
 \end{teo}
\begin{proof}
We wish to solve equation $fh-|a|^2=1/r^4$ with the functions found in Theorem \ref{teo_Kahlerciconia}, since these yield K\"ahler metrics. By analysis of each solution, we find $f,h\in\cinf{r^2}$ and also $a$ constant. And so case (vi) in that Theorem leads the only way forward. Even for the flat base $M$, in case (iii), this yields $f(z)h(r^2)-|a(z)|^2=1/r^4$ for the usual coordinates, and then by computing a few derivatives we get that $f$ and $a$ are constants. Hence we are bound to study case (vi), which implies constant Gauss curvature on the base, cf. Proposition \ref{prop_baseKconstant}. Since $K=-2f'/h$, we shall first study $K=0$. We find $f>0$ constant and then $h(r^2)=\frac{|a|^2}{f}+\frac{1}{fr^4}$. The metric is complete if the distance to the boundary is infinite. To obtain this, we let $a\neq0$. Since $M$ is compact, the distance may be read on the fibres of $Z$. Taking a radial curve with $g$-unit velocity on a fibre, we find respectively the ciconia metric length from  the zero-section to some point or from some point to the boundary at infinity (we need $a\neq0$):
\[ \int_0(h(r^2))^\frac{1}{2}\dx r\sim\int_0\frac{1}{r^2}\dx r =+\infty,\qquad  \int^{+\infty}(h(r^2))^\frac{1}{2}\dx r\sim\int^{+\infty}{1}\,\dx r =+\infty , \]
as we wished, proving (i) in the present result.

Now let us assume $K\neq0$. Recalling the derivative $\,'$ is with respect to $r^2$, we find the identity
\[ \frac{2f'f}{K}+(|a|^2r^2)'=(\frac{1}{r^2})' . \]
The general solution, with $c_0\in\R$, and the respective $h$ are thus
\[  f^2(r^2)=\frac{-K(|a|^2r^4+c_0r^2-1)}{r^2},\ \qquad h(r^2)=\frac{|a|^2r^4+1}{r^4f} .
 \qquad\qquad (*) \]
We see first the case $a=0$. If $K=-1$, we must have $c_0r^2-1>0$ because $f>0$. Thus $c_0>0$  and we must restrict to the domain $r^2>1/c_0$. However, $h={1}/{r^3\sqrt{c_0r^2-1}}$ is not complete in the direction of infinity. If $K=1$, we may have $c_0<0$ and $0<r<+\infty$. Again $h$ is not of infinite length on the fibres. We may also consider $c_0>0$ and restrict to $0<r^2<1/c_0$. Then the determining distances are
\[ \int_0h^\frac{1}{2}\dx r\sim\int_0\frac{1}{r^{\frac{3}{2}}}\dx r =+\infty,\qquad  \int^{\sqrt{\frac{1}{c_0}}}h^\frac{1}{2}\dx r\sim\int^{\sqrt{\frac{1}{c_0}}}
\frac{1}{(1-\sqrt{c_0}r)^{\frac{1}{4}}}\dx r <+\infty . \]
Hence the associated metric space may be completed in the sense of Cauchy sequences to the \textit{outer} boundary, i.e. not including the zero-section. With this we prove (ii).

Finally let us suppose the constant $a\neq0$. For any $c_0\in\R$, we let
\[  \beta_\pm=\frac{-c_0\pm\sqrt{c_0^2+4|a|^2}}{2|a|^2} . \]
Hence $\beta_-<0<\beta_+$ and $|a|^2r^4+c_0r^2-1=|a|^2(r^2-\beta_-)(r^2-\beta_+)$.
Looking up on our general solutions (*), we see first the case $K=1$. Then we have only $0<r^2<\beta_+$ and the singular points of $h$ will be again in the extremes of the interval. We find
\[ \int_0h^\frac{1}{2}\dx r\sim\int_0\frac{1}{r^{\frac{3}{2}}}\dx r =+\infty,\qquad  \int^{\sqrt{\beta_+}}h^\frac{1}{2}\dx r\sim\int^{\sqrt{\beta_+}}\frac{1}{(\sqrt{\beta_+}-r)^{\frac{1}{4}}}\dx r <+\infty  \]
For $K=-1$, we must have $r^2>\beta_+$. The determining distances to the boundary are
\[ \int_{\sqrt{\beta_+}} h^\frac{1}{2}\dx r\sim\int_{\sqrt{\beta_+}}\frac{1}{(r-\sqrt{\beta_+})^\frac{1}{4}}\dx r <+\infty,\quad  \int^{+\infty}h^\frac{1}{2}\dx r\sim\int^{+\infty}\frac{1}{r^{\frac{1}{2}}}\,\dx r =+\infty , \]
showing any Cauchy sequence converges on the space $r^2\geq\beta_+$.
\end{proof}

The $\SU(2)$-holonomy spaces found above are all non-compact and, the majority, are not geodesically complete. Metric space completions for complex manifolds with boundary are interesting in their own. The present examples of true K\"ahlerian and Ricci-flat metrics, the analytical properties of Calabi-Yau manifolds, remain.

%% file: Proofs.tex
{\bf\Large{Some proofs --- not to be included in publication}}

Here we show some easy, though lengthy, proofs of some of the results in the text.

\vspace{1mm}

\begin{figure}[h]
 \centering\includegraphics[width=6cm,keepaspectratio=0.9]{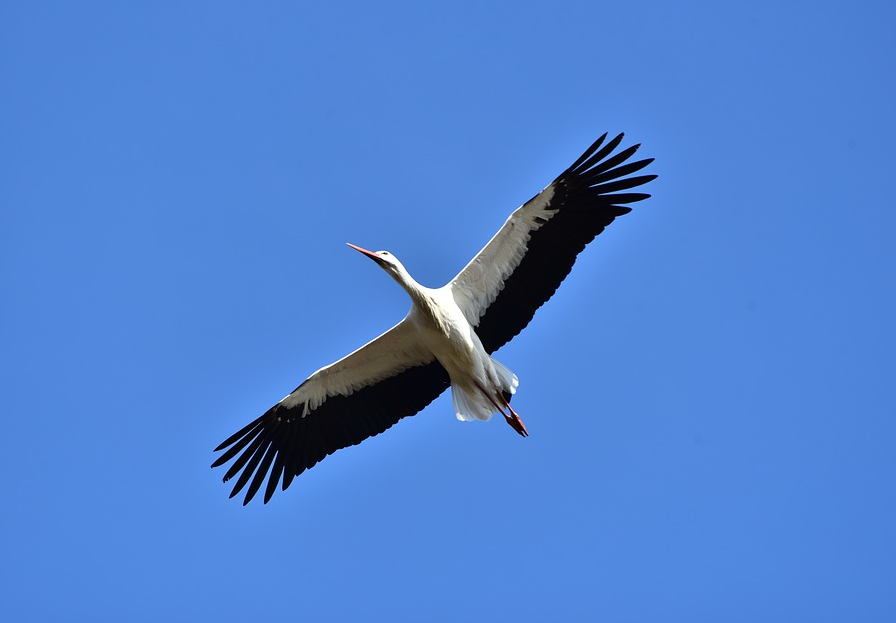}
 \caption{\textit{Ciconia ciconia}}
 \label{figure}
 \end{figure}

\begin{proof}[Proof of Proposition \ref{prop_somebasederivatives}]
Applying the two generators $\partial_z,\partial_\cz$ of $TM^\C$ on the usual definitions, we obtain: $\na_z\dx z=-\Gamma_1\dx z-\Gamma_3\dx\cz$ and $\na_z\dx\cz=-\Gamma_2\dx z-\Gamma_4\dx\cz$. Next we prove \eqref{Gammasfinalmente} by seeing $\na_zg=\na_\cz g=\overline{\na_zg}=0$ is the same as
\[ \papa{\lambda}{z}\,\dx z\dx\cz-\lambda\,(\Gamma_1\dx z+\Gamma_3\dx\cz)\dx\cz 
   -\lambda\,\dx z(\Gamma_2\dx z+\Gamma_4\dx\cz)=0 . \]
Equivalently,
\[   \papa{\lambda}{z}-\lambda\Gamma_1-\lambda\Gamma_4=0,\ \Gamma_2=0,\ \Gamma_3=0 , \]
and, since $\Gamma_4=\overline{\Gamma}_3$, the result follows.
\end{proof}

\vspace{1mm}

\begin{proof}[Proof of Proposition \ref{prop_somederivatives}]
To prove each of these identities one proceeds just as with the following case:
\begin{equation*} \na^*_z\dx z(\begin{cases}
                  \partial_z \\ \partial_\cz \\ \partial_w \\ \partial_\cw 
                 \end{cases}  )= -\dx z(\na^*_z\begin{cases}
                  \partial_z \\ \partial_\cz \\ \partial_w \\ \partial_\cw 
                 \end{cases}  )= -\dx z(\begin{cases}
                  \Gamma\partial_z+w\papa{\Gamma}{z}\partial_w \\ \cw\papa{\cgamma}{z}\partial_\cw \\ \Gamma\partial_w \\ 0 
                 \end{cases}  )= \begin{cases}
                  -\Gamma \\ 0 \\ 0 \\ 0
                 \end{cases}.
\end{equation*}
Implying $\na^*_z\dx z=-\Gamma\dx z$. Another case:
\begin{equation*} \na^*_z\dx\cw(\begin{cases}
                  \partial_z \\ \partial_\cz \\ \partial_w \\ \partial_\cw 
                 \end{cases}  )= -\dx\cw(\qquad \cdots \qquad )= -\dx\cw(\qquad\cdots\qquad )= \begin{cases} 
                  0\\ -\cw\papa{\cgamma}{z} \\ 0 \\ 0
                 \end{cases}.
\end{equation*}
Implying $\na^*_z\dx\cw=-\cw\papa{\cgamma}{z}\dx\cz$. Of course we may also obtain such coefficients from the skew-transpose metric matrix.
\end{proof}